\documentclass[12pt,reqno]{amsart}
\usepackage{amsmath}
\usepackage{amssymb}
\usepackage[left=2.54cm,top=2.54cm,right=2.54cm,bottom=2.54cm]{geometry}
\usepackage{epsfig}
\usepackage{color}                    

\begin{document}
\newtheorem{theorem}{Theorem}[section]
\newtheorem{lemma}[theorem]{Lemma}
\newtheorem{claim}[theorem]{Claim}
\newtheorem{definition}[theorem]{Definition}
\newtheorem{conjecture}[theorem]{Conjecture}
\newtheorem{proposition}[theorem]{Proposition}
\newtheorem{algorithm}[theorem]{Algorithm}
\newtheorem{corollary}[theorem]{Corollary}
\newtheorem{observation}[theorem]{Observation}
\newtheorem{remark}[theorem]{Remark}
\newtheorem{problem}[theorem]{Open Problem}
\newcommand{\noin}{\noindent}
\newcommand{\ind}{\indent}
\newcommand{\al}{\alpha}
\newcommand{\om}{\omega}
\newcommand{\R}{{\mathbb R}}
\newcommand{\N}{{\mathbb N}}
\newcommand\eps{\varepsilon}
\newcommand{\E}{\mathbb E}
\newcommand{\Prob}{\mathbb{P}}
\newcommand{\pl}{\textrm{C}}
\newcommand{\dang}{\textrm{dang}}
\renewcommand{\labelenumi}{(\roman{enumi})}
\newcommand{\bc}{\bar c}
\newcommand{\G}{{\mathcal{G}}}
\newcommand{\D}{{\mathcal{D}}}
\newcommand{\expect}[1]{\E \left [ #1 \right ]}
\newcommand{\floor}[1]{\left \lfloor #1 \right \rfloor}
\newcommand{\ceil}[1]{\left \lceil #1 \right \rceil}
\newcommand{\of}[1]{\left( #1 \right)}
\newcommand{\set}[1]{\left\{ #1 \right\}}
\newcommand{\abs}[1]{\left| #1 \right|}
\newcommand{\angs}[1]{\left\langle #1 \right\rangle}
\newcommand{\sqbs}[1]{\left[ #1 \right]}
\newcommand{\sm}{\setminus}
\newcommand{\bfrac}[2]{\of{\frac{#1}{#2}}}
\renewcommand{\k}{\kappa}
\renewcommand{\l}{\ell}
\renewcommand{\b}{\beta}
\renewcommand{\a}{\alpha}
\renewcommand{\o}{\omega}
\newcommand{\Bin}{\textrm{Bin}}
\newcommand{\Cc}{{\mathcal{C}}}
\newcommand{\Zz}{{\mathcal{Z}}}
\newcommand{\Aa}{{\mathcal{A}}}
\newcommand{\Rr}{{\mathcal{R}}}
\newcommand{\Ee}{{\mathcal{E}}}
\newcommand{\jlo}{2j\frac{\log (n/j)}{\log j}}
\newcommand{\opoo}{(1+o(1))}
\renewcommand{\i}{\ell}

\title{Rainbow arborescence in random digraphs}

\author{Deepak Bal}
\address{Department of Mathematics, Miami University, Oxford, OH, 45056, U.S.A.}
\email{baldc@miamioh.edu}

\author{Patrick Bennett}
\address{Department of Computer Science, University of Toronto, Toronto, ON, Canada, M5S 3G4}
\email{patrickb@cs.toronto.edu}

\author{Colin Cooper}
\address{Department of Informatics, KingÕs College, University of London, London WC2R 2LS, UK}
\email{colin.cooper@kcl.ac.uk}

\author{Alan Frieze}
\address{Department of Mathematical Sciences, Carnegie Mellon University, 5000 Forbes Av., 15213, Pittsburgh, PA, U.S.A}
\thanks{The fourth author is supported in part by NSF grant DMS1362785}
\email{alan@random.math.cmu.edu}

\author{Pawe\l{} Pra\l{}at}
\address{Department of Mathematics, Ryerson University, Toronto, ON, Canada, M5B 2K3}
\thanks{The fifth author is supported in part by NSERC and Ryerson University}
\email{\texttt{pralat@ryerson.ca}}

\maketitle
\begin{abstract}
We consider the Erd\H{o}s-R\'enyi random directed graph process, which is a stochastic process that starts with $n$ vertices and no edges, and at each step adds one new directed edge chosen uniformly at random from the set of missing edges. Let $\D(n,m)$ be a graph with $m$ edges obtained after $m$ steps of this process. Each edge $e_i$ ($i=1,2,\ldots, m$) of $\D(n,m)$ independently chooses a colour, taken uniformly at random from a given set of $n(1 + O( \log \log n / \log n)) = n (1+o(1))$ colours. We stop the process prematurely at time $M$ when the following two events hold: $\D(n,M)$ has at most one vertex that has in-degree zero and there are at least $n-1$ distinct colours introduced ($M= n(n-1)$ if at the time when all edges are present there are still less than $n-1$ colours introduced; however, this does not happen asymptotically almost surely). The question addressed in this paper is whether $\D(n,M)$ has a rainbow arborescence (that is, a directed, rooted tree on $n$ vertices in which all edges point away from the root and all the edges are different colours). Clearly, both properties are necessary for the desired tree to exist and we show that, asymptotically almost surely, the answer to this question is ``yes''.
\end{abstract}

\section{Introduction and the main result}
A set of edges $S$ in an edge-coloured graph is said to be \textbf{rainbow coloured} if every edge in $S$ has a different colour. There has recently been a deal of research on the existence of rainbow coloured objects in random graphs. The existence of rainbow coloured Hamilton cycles is studied in Cooper and Frieze~\cite{CF} and in Frieze and Loh~\cite{FL}. In the latter paper, it is shown that if $G_{n,m}$, a random graph on $n$ vertices and $m=(\frac{1}{2}+o(1)) n\log n$ edges, is randomly coloured with $n+o(n)$ colours, then a.a.s.\ there is a rainbow Hamilton cycle.\footnote{All asymptotics throughout are as $n \rightarrow \infty$ (we emphasize that the notations $o(\cdot)$ and $O(\cdot)$ refer to functions of $n$, not necessarily positive, whose growth is bounded). We say that an event in a probability space holds \textbf{asymptotically almost surely} (or \textbf{a.a.s.}) if the probability that it holds tends to $1$ as $n$ goes to infinity.} In a related paper, Bal and Frieze~\cite{BF} showed that if $m=Kn\log n$ and  if $G_{n,m}$ is randomly coloured with exactly $n$ colours then a.a.s.\ there is a rainbow Hamilton cycle. Janson and Wormald~\cite{JW} showed that if a random $2r$ regular graph $\G_{n,2r}$ on vertex set $[n]$ is randomly coloured so that each colour is used exactly $r$ times then a.a.s.\ $\G_{n,2r}$ contains a rainbow Hamilton cycle. More recently, Ferber, Kronenberg, Mousset and Shikhelman~\cite{FKMS} showed that if $np\gg \log n$ and the binomial random graph $G_{n,p}$ is randomly edge-coloured using $O(n)$ colours then a.a.s.\ $G_{n,p}$ contains $\Omega(np)$ edge disjoint rainbow Hamilton cycles.

\medskip

Hamilton cycles are just one example of a well-studied spanning subgraph. Frieze and McKay \cite{FM} considered the following process. Let $e_1,e_2,\ldots,e_N,N=\binom{n}{2}$ be a random ordering of the edges of the complete graph $K_n$. Let $G_m=([n],E_m=\set{e_1,e_2,\ldots,e_m})$. Suppose that the edges of $K_n$ are randomly coloured from a set of $cn,c\geq 1$ colours. Then let 
\begin{align*}
m_1&=\min\set{m:G_m\text{ is connected}}.\\
m_2&=\min\set{m:E_m\text{ contains edges of $n-1$ different colours}}.\\
m_3&=\min\set{m:G_m\text{ contains a rainbow spanning tree}}.
\end{align*}
It is clear that $m_3\geq \max\set{m_1,m_2}$ and \cite{FM} shows that a.a.s.\ $m_3=\max\set{m_1,m_2}$. Bal, Bennett, Frieze and Pralat~\cite{BBFP} proved a similar tight result for the case where each edge receives a choice of two random colours.

\bigskip

In this paper we turn our attention to a directed analogue of the result of \cite{FM}. We consider a \textbf{random digraph process}, which is a stochastic process that starts with $n$ vertices and no edges, and at each step adds one new (directed) edge chosen uniformly at random from the set of missing edges. Formally, let $N=n(n-1)$ and let $e_1, e_2, \ldots, e_N$ be a random permutation of the edges of the complete digraph $\vec{K}_n$. The graph process consists of the sequence of random digraphs $( \D(n,m) )_{m=0}^{N}$, where $\D(n,m)=(V,E_m)$, $V = [n] := \{1, 2, \ldots, n\}$, and $E_m = \{ e_1, e_2, \ldots, e_m \}$. It is clear that $\D(n,m)$ is a digraph taken uniformly at random from the set of all digraphs on $n$ vertices and $m$ edges. (See, for example,~\cite{bol, JLR} for more details.)

Let 
$$\eps = \eps(n) = \frac{\log \log n}{\log n}.$$
Suppose that each edge $e_i$ ($i=1,2,\ldots, m$) of $\D(n,m)$ is independently assigned a colour, uniformly at random, from a given set $W$ of $(1+50\eps)n = (1+o(1))n$ colours. 

We often write $\D(n,m)$ when we mean a graph drawn from the distribution $\D(n,m)$.  

\bigskip

An \textbf{arborescence} is a digraph in which, for a vertex $u$ called the \textbf{root} and any other vertex $v$, there is exactly one directed path from $u$ to $v$. Equivalently, an arborescence is a directed, rooted tree in which all edges point away from the root. Every arborescence is a directed acyclic graph (DAG), but not every DAG is an arborescence. A set of edges $S$ is said to be \textbf{rainbow} if each edge of $S$ is in a different colour. An arborescence is said to be rainbow if its edge set is. We are concerned with the following four events: 
\begin{eqnarray*}
\Cc_m &=& \{ \text{$\D(n,m)$ contains edges in at least $n-1$ colours} \},\\
\Zz_m &=& \{ \text{At most one vertex of $\D(n,m)$ has in-degree zero} \}, \\
\Aa_m &=& \{ \text{$\D(n,m)$ has an arborescence} \},\\
\Rr_m &=& \{ \text{$\D(n,m)$ has a rainbow arborescence} \}.
\end{eqnarray*}
Let $\Ee_m$ stand for one of the above four sequences of events and let 
$$
m_{\Ee} = \min\{m \in \N : \Ee_m \text{ occurs}\},
$$
provided that such an $m$ exists. (Note that $m_{\Zz}$ and $m_{\Aa}$ are always defined but the other two might not be.) It is obvious that $m_{\Aa} \ge m_{\Zz}$. Moreover, if $m_{\Rr}$ is defined, then so is $m_{\Cc}$ and clearly
$$
m_{\Rr} \ge \max\{m_{\Aa}, m_{\Cc} \} \ge m_{\Zz}.
$$ 

It is known (and easy to show using, say, the Brun's sieve---see, for example, Section 8.3 in~\cite{AS}) that the sharp threshold for property $\Zz$ is $n \log n$; in fact, the random variable counting the number of vertices of in-degree zero tends to the Poisson random variable with  expectation $e^{-c}$, provided that $m = n (\log n + c), c \in \R$, and so
\begin{equation}\label{eq:thr_in-degree}
\Prob ( \Zz_m ) = \of{1+e^{-c}+o(1)} e^{-e^{-c}}.
\end{equation}

Moreover, given other ``hitting time'' results in random graph processes (see, for example ~\cite{bol84, bt85} or Chapters 7 and 8 in \cite{bol}) it is natural to expect that a.a.s.\ $m_{\Aa} = m_{\Zz}$. To the best of our knowledge, this result is not published anywhere, but is implied by our result. Moreover, it follows from the coupon collector problem that, say, a.a.s.\ $m_{\Cc} < \frac n2 \log n$ and so the desired condition for the number of distinct colours present is satisfied much earlier. Indeed, the expected number of colours not present at time $\frac n2 \log n$ is equal to
$$
(1+50\eps) n \left( 1 - \frac {1}{(1+50\eps)n} \right)^{\frac n2 \log n} < 2 n \exp \left( - \frac {\log n}{2(1+50\eps)} \right) = n^{1/2+o(1)}.
$$
Hence, a.a.s.\ at most $n^{2/3} \le 50 \eps n$ colours are missing at that point of the process.

\bigskip

In this paper, we show the following result. 
\begin{theorem}\label{thm:main}
We have that a.a.s.
$$
m_{\Rr} = m_{\Aa} = m_{\Zz}.
$$ 
\end{theorem}
The following corollary follows immediately from~(\ref{eq:thr_in-degree}).
\begin{corollary}
Let $m = n (\log n + c)$ for some $c \in \R$. Then,
$$
\Prob (\Rr_m) = \of{1+e^{-c}+o(1)} e^{-e^{-c}}.
$$
\end{corollary}

\bigskip

Our results refer to the random graph process. However, it will be sometimes easier to work with the $D(n,p)$ model instead of $\D(n,m)$. The \textbf{random digraph} $D(n,p)$ consists of the probability space $(\Omega, \mathcal{F}, \Prob)$, where $\Omega$ is the set of all digraphs with vertex set $\{1,2,\dots,n\}$, $\mathcal{F}$ is the family of all subsets of $\Omega$, and for every $G \in \Omega$,
$$
\Prob(G) = p^{|E(G)|} (1-p)^{n(n-1) - |E(G)|} \,.
$$
This space may be viewed as the set of outcomes of $n(n-1)$ independent coin flips, one for each ordered pair $(u,v)$ of vertices, where the probability of success (that is, adding directed edge $(u,v)$) is $p.$ Note that $p=p(n)$ may (and usually does) tend to zero as $n$ tends to infinity.  We often write $D(n,p)$ when we mean a graph drawn from the distribution $D(n,p)$.  

\bigskip

Lemma~\ref{lem:gnp_to_gnm} below provides us with a tool to translate results from $D(n,p)$ to $\D(n,m)$---see the first Proposition in \cite{pit}. 

\begin{lemma}\label{lem:gnp_to_gnm}
Let $P$ be an arbitrary property and set $p =p(n) = m/{n(n-1)}$. If $m=m(n)\to \infty$ is any function such that $m(1-p) \to\infty$, then,
$$
\Prob(\D(n,m) \in P) \le 5 \sqrt{m} \cdot \Prob(D(n,p) \in P).
$$
\end{lemma}

We will also use the following version of Chernoff bound:
\begin{lemma}[\textbf{Chernoff Bound}] 
If $X$ is a binomial random variable with expectation $\mu$, and $0<\delta<1$, then 
$$
\Pr[X < (1-\delta)\mu] \le \exp \left( -\frac{\delta^2 \mu}{2} \right)
$$ 
and if $\delta > 0$,
\[\Pr\sqbs{X > (1+\delta)\mu} \le \exp\of{-\frac{\delta^2 \mu}{2+\delta}}.\]
\end{lemma}

\section{Proof of Theorem~\ref{thm:main}}\label{sec:proof}

\subsection{Overview}

The proof is technical and involves revealing some aspects of the random object one by one. We give this overview to summarize the order in which randomness is revealed in the proof. 

Before the process begins we arbitrarily select $5\eps n$ colours to be ``special''. The remaining colours are ``regular''. Now, we consider the random graph process, and for each edge we reveal whether its colour is special or regular. If the edge is special, then we do not reveal anything else at this time; on the other hand, if the edge is regular, then we reveal its head vertex (but not its tail and not its colour).

After $m_-$ many steps of the random graph process, we identify some ``dangerous" vertices that have in-degree at most 2 (in regular colours).  Now, for these dangerous vertices, we expose their in-degree in special edges. We then continue the process until time $m_\Zz$ revealing the heads of regular coloured edges and revealing the heads of special edges directed to dangerous vertices. By doing this, we are able to determine exactly when the event $\Zz_m$ holds for the first time (\emph{i.e.} $m_\Zz$) since the unique vertex of in-degree zero at time $m_\Zz$ must be among the dangerous vertices.  Actually, this vertex, which we call $u$, obviously must be among the vertices of in-degree zero, but we define dangerous in this way for another reason.  Also note that we do not need to know the special heads to non-dangerous vertices to determine $m_\Zz$. Our goal here is to show that at time $m_{\Zz}$, each vertex (except for $u$, which will become the root of the final arborescence) can select a unique colour from its in-edges. The $n-1$ colours used for this will now be called ``mapping colours''.  Note that the mapping colours are either regular or they are special and belong to an edge directed to a dangerous vertex. Thus we will be left with at least $(1+45\eps )n - (n-1) = 45\eps n +1$ many regular colours which are not mapping colours.

We will then show that a.a.s.\ there is an arborescence of size at least $n^{2/3}$, using only regular non-mapping colours, and rooted at the vertex $u$. At that point we reveal the tails of all the edges of mapping colours. These edges, together with the size $n^{2/3}$ arborescence, will a.a.s.\ give us one large arborescence with almost all of the vertices in the graph, together with a few small arborescences. To show that we can connect them all together, we finally reveal the heads and tails of the remaining unrevealed special edges. 

\subsection{Proof}

Let $\omega = \omega(n)$ be any function tending to infinity (but sufficiently slowly) together with $n$. For definiteness, let $\omega = \omega(n) = \log \log n$. Let
$$
m_- := \left\lfloor n ( \log n - \omega ) \right \rfloor \quad \text{ and } m_+ := \left\lceil n ( \log n + \omega ) \right \rceil.
$$
(Here we make sure $m_-$ and $m_+$ are both integers but later on expressions such as $n ( \log n + \omega )$ that clearly have to be an integer, we round up or down but do not specify which: the choice of which does not affect the argument.) As we already mentioned, $m_{\Rr} \ge m_{\Aa} \ge m_{\Zz}$ and it is known that a.a.s.
\begin{equation}\label{eq:range_for_m}
m_- \le m_{\Zz} \le m_+.
\end{equation}
Our goal is to show that at time $m_{\Zz}$, a.a.s.\ there is a rainbow arborescence in $\D(n,m_{\Zz})$; that is, that a.a.s.\ $m_{\Rr} = m_{\Zz}$.

\bigskip

Let us start with showing that at time $m_{\Zz}$, a.a.s.\ all but the only vertex of in-degree zero can select a unique colour assigned to an edge from some of their in-neighbours. However, as we mentioned earlier, it is important that we apply the so-called \textbf{multi-round exposure} technique. We arbitrarily choose $5\eps n$ colours to be \textbf{special} and let the other $(1+45\eps)n$ colours be \textbf{regular}.  At this point of the process, we expose only the number of in-neighbours and colours assigned to the regular-coloured associated edges but no in-neighbour is exposed yet. We also reveal the in-degrees in special colours of the dangerous vertices (formally defined in the proof of Lemma~\ref{lem:assignment} below), but nothing else. As a consequence of this, if we condition on the event that vertex $v$ has precisely $k$ in-neighbours (and perhaps also on the event that some of the edges from them to $v$ are coloured in the desired way), in-neighbours can be selected uniformly at random from all ${n-1 \choose k}$ possible $k$-tuples. 

\begin{lemma}\label{lem:assignment}
A.a.s.\ the following property holds in $\D(n,m_{\Zz})$.  There exists a set $T \subset W$ of size $n-1$, a vertex $u \in V$, and a perfect matching $f : V \setminus \{u\} \to T$ so that every vertex $v \in V  \setminus \set{u}$ has an edge directed to $v$ in colour $f(v)$.
\end{lemma} 

Note that $u$ must be the only vertex of in-degree zero in $\D(n,m_{\Zz})$. This vertex will be used as a root of the arborescence.  The $n-1$ colours from the set $T$ will now be called \textbf{mapping colours}.

\begin{proof}
We think of the assignment problem we deal with as a bipartite graph: on one side we have the set of $n$ vertices $V$, on the other side we have the set of $(1+50\eps)n$ colours $W$. A vertex $v \in V$ is adjacent to colour $c \in W$ if at least one edge to $v$ has colour $c$. We use Hall's necessary and sufficient condition to show that the desired matching exists. 

There is no matching saturating $V$ if and only if for some $k \ge 1$ there exists a \textbf{$k$-witness}, that is, a pair $(S,T)$ of sets $S \subseteq V$, $T \subseteq W$ such that $|S|=k$, $|T|=k-1$, and $N(S) \subseteq T$. 
We say that vertex $v$ is \textbf{dangerous} if its in-degree (in regular colours) is at most $2$ at time $m_-$. The probability that a given vertex has $\ell \le 2$ regular-coloured edges incident to it is equal to
\begin{align*}
{m_- \choose \ell} & \of{ \frac {n-O(\ell)}{n(n-1)-O(m_-)} \cdot \frac{1+45 \eps}{1+50\eps}}^{\ell} \of{ 1 - \frac{n-O(\ell)}{n(n-1) - O(m_-)}\cdot \frac{1+45 \eps}{1+50\eps} }^{m_- - \ell} \\
& = (1+o(1)) \frac {(n \log n)^{\ell}}{\ell!} \cdot \frac {1}{n^{\ell}} \cdot \exp \of{ - \frac {m_-}{n} \of{ 1-5\eps+O(\eps^2) }} \\
& = (1+o(1)) \frac {(\log n)^{\ell}}{\ell!} \cdot \exp \of{ - \log n + \omega + 5 \eps \log n } \\
& = (1+o(1)) \frac {e^{\omega + 5\eps \log n} (\log n)^{\ell}}{\ell! n}. 
\end{align*}
Hence, the expected number of dangerous vertices is $O(e^{\omega + 5 \eps \log n} \log^{2}  n) = O(\log^{8} n)$ and so a.a.s.\ there are at most, say, $O(\log^{9} n)$ dangerous vertices. We reveal the in-degrees in special colours for these vertices and continue to do so for the rest of the process until time $m_{\Zz}$. Now, it follows immediately from the definition of $m_{\Zz}$ that there is precisely one 1-witness at that point of the process: $S=\{u\}$, where $u$ is the only vertex of in-degree zero. 

Now, our goal is to show that a.a.s.\ there is a matching saturating $V \setminus \{u\}$; that is, a.a.s.\ there is no $k$-witness $(S,T)$ with $S \subseteq V \setminus \{u\}$ and $2 \le k \le n-1$.  First, let us concentrate on the  case $k=2$.  Note that at time $m > m_-$, a dangerous vertex may or may not have at least $3 $ in-neighbours but a non-dangerous vertex will definitely have at least $3 $. Moreover, note that vertex $u$ is dangerous but every other dangerous vertex has in-degree at least 1 at time $m_{\Zz}$. Let us focus on the first incoming edges assigned to these dangerous vertices (regardless of whether they are regular-coloured or special), conditioning on the event that there are $O(\log^{9} n)$ such vertices.  The probability that two of these dangerous vertices receive the same regular colour on their first incoming edge is at most
\[\binom{O(\log^{9}n)}{2}\cdot\of{1+45\eps}n\cdot\of{\frac{1}{(1+45\eps)n}}^2 = O\of{\frac{\log^{18}n}{n}} = o(1).\]
Similarly, the probability that two of these dangerous vertices receive the same special colour on their first incoming edge is at most
\[\binom{O(\log^{9}n)}{2}\cdot 5\eps n\cdot\of{\frac{1}{5 \eps n}}^2 = O\of{\frac{\log^{18}n}{\eps n}} = o(1).\]

So we may assume that at time $m_{\Zz}$, every pair of dangerous vertices (excluding $u$) sees at least 2 different colours. On the other hand, for each non-dangerous vertex we focus on the first $3$ regular-coloured incoming edges assigned to them. The probability that all 3 of these edges receive the same colour is equal to
$$
{(1+45\eps) n } \of{ \frac {1}{(1+45\eps)n} }^3  = O \of{ \frac {1}{n^2} }.
$$
Hence, a.a.s.\ all non-dangerous vertices have at least least $2$ distinct colours assigned at time $m_{\Zz}$. Conditioning on these two events, we get that there is no $2$-witness at time $m_{\Zz}$. Indeed, every vertex in a set $S \subseteq V$ consisting of two dangerous vertices yields a unique colour; on the other hand, $S$ containing at least one non-dangerous vertex yields at least $2$ colours.

Now, we move to the case $k \ge 3$. This time, it will be convenient to focus on minimal configurations. A $k$-witness is \textbf{minimal} if there does not exists $S' \subset S$ and $T' \subset T$ such that $(S',T')$ is a $k'$-witness, where $k' < k$. It is straightforward to see that if $(S,T)$ is a minimal $k$-witness, then every $c$ in $T$ has degree at least 2 in the graph induced by $S \cup T$. Hence there are at least $2(k-1)$ edges between $S$ and $T$. The goal is to show that there is no minimal $k$-witness in $\D(n,m_{\Zz})$, even when only regular colours are taken into consideration. However, having no minimal $k$-witness at time $m_1$ does not imply that there is no $k$-witness at time $m_2  > m_1$. Hence, we need to estimate the probability that the desired property holds for a given $m$ such that $m_- \le m \le m_+$ and then take a union bound over all $O(n \omega) = O(n \log n)$ possible values of $m$. In fact, we will show that with probability $1-o( (n \log n)^{-1} (n \log n)^{-1/2})$, $D(n,p)$ has no $k$-witness for  $k \ge 3$, for 
$$
p_- \le p \le p_+
$$
where
\[p_\pm = \frac {m_\pm}{n(n-1)} = \frac {\log n \pm \omega + o(1)}{n}.\]
As already mentioned, the claim will hold by Lemma~\ref{lem:gnp_to_gnm} and the union bound.

Let 
$$
p':= p \cdot \frac{1+45 \eps}{1+50\eps} = \frac{\log n  -O(\log \log n)}{n}
$$ 
be the probability that an edge is both present and of a regular colour. Fix $S \subseteq V,T \subseteq W$, $|S|=k$, $|T|=k-1$. The probability that the number of in-neighbours (to $S$, in regular colours) is equal to $t$ is
$$
{k(n-1) \choose t} p'^t (1-p')^{k(n-1)-t}.
$$
So the probability that $(S,T)$ is a $k$-witness is at most
\begin{align*}
p_{S,T} &= \sum_{t \ge 2 (k-1)} {k(n-1) \choose t} p'^t (1-p')^{k(n-1)-t} \of{ \frac {k-1}{(1+ 45\eps)n}}^t.
\end{align*}
Thus, the expected number of minimal witnesses for any $k \ge 3$ is at most
\begin{align}
&\sum_{k \ge3} \binom{n}{k}\binom{(1+45\eps)n}{k-1}\sum_{t \ge 2 (k-1)}   {k(n-1) \choose t} p'^t (1-p')^{k(n-1)-t} \of{ \frac {k-1}{(1+45\eps)n}}^t. \label{witnesses}
\end{align}
To estimate the inner sum, first note that the ratio of consecutive terms is equal to
\begin{align*}
\frac{{k(n-1) \choose t+1} p'^{t+1} (1-p')^{k(n-1)-t-1} \of{ \frac {k-1}{(1+45\eps)n}}^{t+1}}{{k(n-1) \choose t} p'^t (1-p')^{k(n-1)-t} \of{ \frac {k-1}{(1+45\eps)n}}^t} &= \frac{k(n-1) - t }{t+1} \cdot \frac{p'}{1-p'} \cdot \frac{k-1}{(1+45\eps)n}\\
& \le \frac{k(n-1) - 2 (k-1) }{2 (k-1)+1} \cdot \frac{p'}{1-p'} \cdot \frac{k-1}{(1+45\eps)n}\\
& \le \frac{k \log n}{1.5n}
\end{align*}
for $n$ sufficiently large. So, when $k < \frac{n}{\log n}$, the inner sum is of the order of its first term. Now, for $k \ge \frac{n}{\log n}$, note that each term is small: 
\begin{align*}
\binom{n}{k} & \binom{(1+45\eps)n}{k-1}{k(n-1) \choose t} p'^t (1-p')^{k(n-1)-t} \of{ \frac {k-1}{(1+45\eps)n}}^t  \\
& \le \bfrac{ne}{k}^k \bfrac{(1+45\eps)ne}{k-1}^{k-1} \bfrac{k(n-1)e}{t}^t \bfrac{p'}{1-p'}^t \exp\{-k(n-1)p'\} \bfrac {k-1}{(1+45\eps)n}^t\\
& \le \exp\left\{-k(n-1)p' + 2 k \log\bfrac{2ne}{k} \right\}  \bfrac{k^2p'e}{t(1-p')(1+45\eps)}^t \\
& \le \exp\left\{-k(n-1)p' + \frac{k^2p'}{(1-p')(1+45\eps)} + 3 k \log \log n \right\},
\end{align*}
for $n$ large enough, since $\bfrac{c}{t}^t \le \exp \left\{ \frac{c}{e} \right\}$ which holds for all positive numbers $c, t$. Indeed, this inequality follows by noting that $\log(x) \le x/e$ for all $x>0$ and letting $x =c/t$. Hence, each term is at most
\begin{align*}
 \exp&\left\{-k(n-1)p' + k n p' (1+O(p'))(1- 45 \eps + O\of{\eps^2} ) + 3 k \log \log n \right\}   \\
&  = \exp\left\{- (45+o(1)) \of{\eps knp'} + 3 k \log \log n \right\} \\
& \le \exp\left\{-41 k \log \log n \right\} = \exp\left\{-\Omega\of{ \frac {n\log \log n}{\log n} } \right\}.
\end{align*}
Therefore, the contribution from these terms is $O(n^2 \exp(-\Omega(n \log \log n / \log n))) = o(n^{-2})$ and \eqref{witnesses} can be upper bounded by
\begin{align*}
&o(n^{-2}) + \sum_{k = 3}^{\frac{n}{\log n}} \sum_{t \ge 2 (k-1)} \binom{n}{k}\binom{(1+45\eps)n}{k-1} {k(n-1) \choose t} p'^t (1-p')^{k(n-1)-t} \of{ \frac {k-1}{(1+45\eps)n}}^t  \\
& = o(n^{-2}) + O\of{ \sum_{k =  3}^{\frac{n}{\log n}}  \binom{n}{k}\binom{(1+45\eps)n}{k-1} {k(n-1) \choose 2(k-1)} p'^{2(k-1)} (1-p')^{k(n-1)-2(k-1)} \of{ \frac {k}{(1+45\eps)n}}^{2(k-1)}}\\
&= o(n^{-2}) +  O\of{ \sum_{k = 3}^{\frac{n}{\log n}} \bfrac{ne}{k}^k\bfrac{(1+45\eps)ne}{k-1}^{k-1} \of{\frac{k(n-1)e}{2(k-1)}\cdot\frac{p'}{1-p'}\cdot\frac{k-1}{(1+45\eps)n}}^{2(k-1)}(1-p')^{k(n-1)} } \\
&= o(n^{-2}) +  O\of{ \sum_{k = 3}^{\frac{n}{\log n}} \exp\of{-(k-1)\log n + O(k\omega + k\log\log n)}}\\
&= o(n^{-2}) +  O\of{ \sum_{k = 3}^{\frac{n}{\log n}} \exp\of{- \opoo(k-1)\log n  } } \\
&= o\of{n^{-2}} + n^{-  2+ o(1)} =o( (n \log n)^{-1} (n \log n)^{-1/2})
\end{align*}
as desired, since the sum on the second to last line is dominated by its first term. The proof of the lemma is finished, as explained earlier.
\end{proof}

\bigskip

Now we will show that some of the regular colours not in $T$ can be used to start building a rainbow arborescence of size at least $n^{2/3}$ rooted at the special vertex $u$. After that, we will show how to convert this arborescence into a spanning one. Let us note that Lemma~\ref{lem:assignment} proves only the existence of a set of $n-1$ colours and a corresponding matching. Unfortunately, we do not control which colours are used and which are still available for us to construct the arborescence around $u$. Therefore, we can only use properties that hold for \emph{all} sets of colours of the desired size.

Let us start with the following useful observation that holds a.a.s.\ in $\D(n,m_-)$, and is deterministically preserved in $\D(n,m)$, provided that $m > m_-$.

\begin{lemma}\label{lem:no_bad_epsn}
A.a.s.\ the following property holds in $\D(n,m_-)$.  Every set of colours $S \subseteq W$ with $|S| = 45 \eps n$ satisfies 
\[\abs{\set{v\in V\,:\, \deg^-_{S}(v) \le  43 \eps\log n}} \le \frac{n}{\eps \log n} = \frac{n}{ \log \log n}\]
where $\deg_S^-(v)$ represents the in-degree of $v$ in the graph induced by edges whose colours are from $S$.
\end{lemma}
\begin{proof}
For convenience, we will work with $D(n,p_-)$ with 
$$
p_- = \frac {m_-}{n(n-1)} = \frac {\log n - \omega}{n-1} = \frac {\log n - \omega - o(1)}{n}.
$$
Let $S$ be any set of colours of size $45\eps n$. Then 
\[
\E\sqbs{\deg^-_S(v)} = \frac{45\eps}{1+50\eps} \ \cdot p_- \cdot (n-1) > 44 \eps \log n.
\]
It follows from Chernoff bound that
\[ \Pr \sqbs{\deg^-_S(v) \le 43 \eps \cdot \log n} \le \exp \left( - \frac{ 1}{2} \cdot \bfrac{1}{44}^2 \cdot \eps \log n \right) \le \exp\of{-\frac{\eps \log n}{5000} }. \]
So the probability that there exist $\frac{n}{\eps \log n}$ many vertices of small in-degree is at most 
\[ 
\binom{n}{\frac{n}{\eps \log n}} \cdot \exp \left( -\frac{\eps \log n}{5000} \cdot \frac{n}{\eps \log n} \right)  \le \exp \left(\frac{n}{\eps \log n} \cdot \log(e \eps \log n) -\frac{n}{5000} \right) \le \exp \of{-\Omega\of{ n}}.
\] 
Since there are 
$$
\binom{(1+50\eps )n}{45\eps n} \le \exp \of{45\eps n \log \bfrac{e \cdot (1+50\eps)}{45\eps}} = \exp \of{ O\of{ \frac {\log^2 \log n}{\log n} n }}  = \exp\of{o(n)}
$$ 
sets of colours to consider, the desired property holds with probability $1-o(1/n)$ in $D(n,p_-)$ by the union bound, and so it holds a.a.s.\ in $\D(n,m_-)$ by Lemma~\ref{lem:gnp_to_gnm}. 
\end{proof}

\bigskip

We will need the following three claims. The first claim is an easy observation. We say a vertex and an edge are incident if the vertex is either the head or tail of the edge.  

\begin{claim}\label{claim:basic_facts}
In $\D(n,m)$ with $m\le m_+$, a.a.s., 
\begin{enumerate}
\item no colour appears more than $10\log n$ times,
\item no vertex is incident to more than $10\log n$ edges,
\item no vertex is incident to more than $10$ edges of the same colour.
\end{enumerate}
\end{claim}
\begin{proof}
It is enough to prove these properties for $D(n,p_+)$ and apply Lemma \ref{lem:gnp_to_gnm}. Each follows easily from Chernoff and union bounds.
\end{proof}
While the above lemma holds for all colours, we will only use it for regular colours and for special colours incident to dangerous vertices.

\bigskip

As we will not have failure probabilities small enough to union bound over all choices of $m$ with $m_-
\le m\le m_+$, we would like to prove the existence of the desired structure at time $m_-$. One issue is that at time $m_-$, we do not know which vertex is going to be the root. However, we do know (as it is shown in the proof of Lemma~\ref{lem:assignment}) that a.a.s.\ there are $O(\log^{9} n)$ many dangerous vertices; that is, vertices of in-degree at most 2 in regular colours.  One of these vertices must be the root $u$ at time $m_{\Zz}$. 

Let $C'$ be an arbitrarily chosen set of $45 \eps n$ regular non-mapping colors. (Recall that there are at least $45\eps n +1$ many of them to choose from.  By Lemma~\ref{lem:no_bad_epsn}, all but at most $n/  \log \log n$ many vertices $v$ have $\deg_{C'}^-(v) \ge 43 \eps\log n$. We will call such vertices \textbf{good}.  

\begin{claim}\label{claim:pot_roots}
In $\D(n,m_-)$, a.a.s.\ every dangerous vertex $w$ has $\deg^+_{C'}(w) \ge 10 \eps\log n = 10 \log \log n$.
\end{claim}
\begin{proof}
For any good vertex $v$ and any vertex $w$, we have that 
\[\Prob\sqbs{(w,v) \in E_{C'}} = \frac{\deg^-_{C'}(v)}{n-1} \ge \frac{42\eps\log n}{n},\]
where $E_{C'}$ represents the edges coloured with $C'$.
Thus we have that for every vertex $w$, $\deg^+_{C'}(w)$ stochastically dominates 
\[\Bin\sqbs{n\of{1-\frac{1}{ \log  \log n}} , \frac{42 \eps\log n}{n}},\]
with the expected value of $(42+o(1)) \eps \log n$.
So by Chernoff Bound,
$$
\Prob\sqbs{\deg^+_{C'}(w) < 10 \eps\log n} \le  \exp \of{ -\frac12 \cdot \bfrac{42-10}{42}^2 \cdot (42+o(1)) \eps \log n } = o(\log^{-9} n).
$$
Thus, a.a.s.\ each dangerous vertex has at least $10 \eps \log n$ many out neighbours since we again condition on the event that there are only $O(\log^{9} n)$ of them. 
\end{proof}

\bigskip

We will show that, if $S$ is not too small nor too large, there are many edges outgoing from $S$ with different heads and different colours. This will be enough to grow the tree to the desired size.

\begin{claim}\label{claim:S_expands}
Let $S$ be a set of vertices such that $|S| = \Omega(\eps\log n) = \Omega(\log \log n)$, $|S|\le n^{2/3}$, and let $T$ be a set of at least $n/2$ good vertices which is disjoint from $S$. Then with probability at least $1-\exp\of{-\Omega(\eps^2 \log^2n)}  = 1-\exp\of{-\Omega(\log^2 \log n)}$ the following holds: there is a set of $3|S| \eps \log n = 3 |S| \log \log n$ edges going from $S$ to $T$ whose head vertices are all distinct and whose colours are all distinct and in $C'$.
\end{claim}
\begin{proof}
We use a two-round exposure. First we will expose whether some edges are present and in the colour set $C'$, without exposing exactly which colour they are. In the second round we take the edges we found in the first round and expose their colours, conditioning on the colours being in $C'$ of course. Let $v$ be any good vertex and set $d=\deg_{C'}^-(v)$. Suppose for now that $S$ is any set of vertices not including $v$ and let $|S|=s$. Then the probability that there is an edge from a vertex of $S$ to $v$ is
\begin{align*}
1-\frac{\binom{n-1 - s}{d}}{\binom{n-1}{d}} &= 1 - \frac{(n-1-s)_d}{(n-1)_d} \\
&= 1 - \prod_{i=1}^{d} \of{1-\frac{s}{n-i}}\\
&\ge 1-\of{1-\frac{s}{n}}^{d}\\
&\ge 1-\of{1-\frac{s}{n}}^{43 \eps \log n}.
\end{align*}
Now  suppose $S$ satisfies $s = \Omega(\eps\log n)$ and $s\le n^{2/3}$. Then $|N_{C'}^+(S) \cap T|$ (where $N_{C'}^+(S)$ represents the out-neighbours of vertices in $S$ coloured with colours from $C'$) stochastically dominates 
\[\Bin\sqbs{\frac n3, 1-\of{1-\frac{s}{n}}^{43\eps \log n}}.\]
Since $s\le n^{2/3}$, we have
\[ 1-\of{1-\frac{s}{n}}^{43\eps \log n} \sim \frac{43 \eps s \log n}{n}.\]
So by the Chernoff Bound, the probability that $|N_{C'}^+(S) \cap T|$ less than $10 s\eps\log n$ is at most
$\exp\of{-\Theta\of{s\eps\log n}} \le \exp\of{-\Omega(\eps^2 \log^2n)}.$

Now we do the second round of exposure for the proof of the claim. Conditioning on having $10s\eps \log n$ edges going from $S$ to $T$ with distinct heads and colours in $C'$, we reveal which colours the edges are to see if they are distinct. The probability that 4 of these edges get the same colour is at most
\[\binom{10s \eps \log n}{4} \cdot 45 \eps n \cdot \bfrac{1}{45 \eps n}^4 \le n^{(2/3) \cdot 4 + 1 - 4 +o(1)} =  n^{-1/3 + o(1)}\le \exp\of{-\Omega(\eps^2 \log^2n)},\]
where we have used the fact that $s \le n^{2/3}$.
We conclude that there are at least $3s\eps \log n$ distinct colours on these edges and the proof of the claim is finished.
\end{proof}

\bigskip

With these three claims, we can begin to grow our arborescence.   

\begin{lemma}
A.a.s.\ $\D(n,m_{\Zz})$ has a rainbow arborescence with $n^{2/3}$ many vertices, rooted at vertex $u$ which uses only colours from $C'$.
\end{lemma}

\begin{proof}
As mentioned, we will actually show that this statement holds at time $m_-$ for all dangerous vertices, of which $u$ is one.  For each dangerous vertex $w$ we can do the following.  Let $B$ initially be the set of vertices which are not good, (that is, those which have $\deg_{C'}^{-}(v) < 43 \eps \log n$) so $|B| \le \frac{n}{\eps \log n}$ by Lemma~\ref{lem:no_bad_epsn}. As we proceed, we will want to ignore certain vertices and we will move them into $B$. Note that some of the dangerous vertices might be in $B$ (but that does not affect our argument).  Let $T$ be the remaining good vertices that are not dangerous.  By Claims~\ref{claim:basic_facts}(iii) and~\ref{claim:pot_roots}, $w$ has at least $\eps\log n$ many out-neighbours with distinct colours from $C'$. Let $C''$ be a  set of  precisely $\eps\log n$ of these colours and let $S$ be the corresponding out-neighbours. Move from $T$ to $B$ any vertex which is incident to an edge with a colour from $C''$. Then by Claim~\ref{claim:basic_facts}(i), we have removed at most $O(\log^2 n)$ vertices from $T$. We may now repeatedly apply Claim~\ref{claim:S_expands} to find edges from $S$ to $T$ with  new colours. After each application of Claim~\ref{claim:S_expands}, we replace $S$ by the set of heads of the newly found edges and we move from $T$ to $B$ all vertices incident to an edge with the same colour as any of the new edges. We stop when we have first accumulated at least $n^{2/3}$ vertices.  At most $O(n^{2/3}\log n) = o(n)$ many vertices are removed from $T$ and thus we always have $|T|\ge n/2$. Each application of Claim~\ref{claim:S_expands} multiplies the number of vertices in our arborescence by a factor of $\Omega\of{\eps \log n} = \Omega(\log \log n)$, so we apply it $O(\log n / \log \log \log n) = O(\log n)$ many times (to each of $O(\log^{9} n)$ dangerous vertices). We succeed a.a.s.\ by the union bound since Claim \ref{claim:S_expands} holds with probability $1 - \exp\of{-\Omega(\log^2 \log  n)}$  and
\[O(\log^{10} n \exp\of{-\Omega(\log^2 \log n)} ) = \exp\of{-\Omega(\log^2 \log n)} = o(1). \]
The proof of the lemma is finished. 
\end{proof}

\bigskip

Now we expose the  $n-1$ edges we reserved at the beginning, associated with mapping colours. In the digraph induced by these $n-1$ edges, every vertex has in-degree $1$ except for the root, which has in-degree $0$. Thus if we randomly put an extra edge pointing to the root, we have a ``loop-less random mapping'' (which is actually independent from the arborescence we have built so far).

We call the reader's attention to a few small but important points. First, a random mapping is ordinarily thought of as a digraph where each vertex chooses a random \emph{out-neighbour}. We have the reverse situation. Each vertex is choosing a random \emph{in-neighbour}. Second, a pure random mapping allows vertices to map to themselves, \emph{i.e.} allows self-loops. We have no loops in our situation. However, in a random mapping, the number of loops is asymptotically distributed as Poisson with mean 1. Thus with probability bounded away from 0, a random mapping has no loops and any properties which hold a.a.s.\ for random mappings also hold a.a.s.\ for loop-less random mappings.

We now consider the so called ``inverse epidemic process'' (See \cite{g77} or Section 14.5 in \cite{bol}). The process is as follows: suppose that $m$ members of a population are initially infected with a disease. The disease spreads to other vertices along the edges of a random mapping in the reverse orientation. Note that for us, since our mapping edges are already reversed, this means that the disease is spreading along our digraph edges in the forward direction.

We would like to apply a result of Burtin~\cite{burtin}. The result implies that if $x = n^{2/3}$  vertices are initially infected, and $\eta$ is the random variable for the number of eventually infected vertices, then $\bfrac{x}{n}^2 (n-\eta)$ is a.a.s.\ less than any function going to infinity with $n$, say $n^{1/6}$. So all but at most, say $n^{5/6}$ vertices total are eventually infected. We let our initially infected set, $A$, be the vertices in the arborescence we found of size $n^{2/3}$. The eventually infected vertices are in the same component as the root vertex. Denote the eventually infected vertices as $M_I$ and the uninfected vertices as $M_U$. First, we would like to show that a.a.s. $M_U$ contains no dangerous vertices.  Since our random mapping is independent of $A$, we have that any vertex in $V\setminus A$ is equally likely to be in $M_U$ and this probability is $|M_U| / |V\setminus A| = O\of{n^{5/6}/(n - n^{2/3})} = O(n^{-1/6})$. Thus the expected number of dangerous vertices in $M_U$ is $O(\log^9n\cdot n^{-1/6}) = o(1)$ and so a.a.s. $M_U$ contains no dangerous vertices.

Now, it is well known and easy to see that a.a.s.\ there are $O(\log n)$ many cycles in a random mapping. Here we consider double edges to be cycles and note that each component of the digraph induced by the mapping edges is unicyclic. We arbitrarily remove one edge from each cycle to get $O(\log n)$, many arborescence components, each with its own ``root''. Our goal is to show that we can connect the large component $M_I$ to the small arborescence components, whose vertices comprise $M_U$.

Note that  since there are  $O(\log^9 n)$ dangerous vertices, each incident to at most $10 \log n$ colours, there are at least $5\eps n - O(\log^{10} n)$ special, non-mapping colours. We will use these   colours to connect the big component to the small ones. Recall that so far we have exposed all regular edges, as well as the in-degrees in special edges of the dangerous vertices (and these are the only special in-degrees which have been revealed). 
Suppose $A_1$ is a small arborescence component (contained in $M_U$) and suppose its root is $u_1$. To connect the large component to $A_1$, we search for an edge whose tail is anywhere in $M_I$, whose head is $u_1$, and whose colour has not yet been used. Since no vertex of $M_U$ is dangerous, $u_1$'s special in-degree is unrevealed.  If we find a special edge from the large component to $u_1$, we add it to the arborescence (and discard that colour from future use).  We will attempt to do this for $O(\log n)$ many small components. Hence, at most $O(\log n)$ colours are discarded this way, and since each colour appears on at most $10 \log n$ edges, at most $O(\log^2 n)$ edges are discarded. A.a.s.\ there are at least $(5-o(1))\eps n \log n$ many edges in the set of eligible colours when we start exposing edges, so there are at least $4\eps n \log n$ edges to use each time. Also, by Claim \ref{claim:basic_facts}(ii) we may assume that no vertex (and so none of the roots) has in-degree more than $10 \log n$. So when we look for our random special edge going to a particular root, there are at most $10 \log n$ possibilities forbidden by our conditioning. Thus if there are $x$ many small components that remain to be connected to the giant, then the probability that any of them fail to connect is at most  
\[ x\of{1 - \frac{n-n^{5/6} - 10 \log n}{n(n-1)  }}^{4\eps n \log n} \le x\of{1-\frac{1}{2n} }^{4\eps n \log n} \le x \exp\left\{- 2 \eps \log n \right\} = \frac {x}{\log^2 n} =o(1), \]
since $x = O(\log n)$. This finishes the proof of Theorem~\ref{thm:main}.
\section{Conclusion and open questions}
Our result, Theorem \ref{thm:main}, is tight in one sense in that the hitting time is correct in terms of the structure of the digraph. It would be of interest to see precisely how many colours we need to choose from to get this result. Can we find the precise hitting time if we only have $n-1$ colours? Suppose we want a rainbow copy of some specific spanning tree, for example, a directed Hamilton path. How many random edges and colours are needed? How many edge disjoint rainbow spanning trees (respectively, arborescences) can we pack into $G_{n,m}$ (respectively, $\D_{n,m}$)?

\end{document}